\documentclass[10pt]{amsart}

\usepackage{amsfonts,amssymb,amscd,amstext}
\usepackage[paper size={210mm,297mm},left=34mm,right=34mm,top=43.5mm,bottom=43.5mm]{geometry}
\usepackage[charter]{mathdesign}
\usepackage[colorlinks=true,linkcolor=blue,citecolor=red]{hyperref}
\usepackage{graphicx}
\usepackage[utf8]{inputenc}
\usepackage{esint}
\usepackage{enumerate}

\renewcommand{\leq}{\leqslant}
\renewcommand{\geq}{\geqslant}

\renewcommand{\ge}{\geqslant}
\newcommand{\ptl}{\partial}

\newcommand{\rr}{{\mathbb{R}}}
\newcommand{\cc}{{\mathbb{C}}}

\newcommand{\la}{\lambda}

\newcommand{\hh}{{\mathbb{H}}}
\newcommand{\nn}{{\mathbb{N}}}

\newcommand{\sub}{\subset}

\newcommand{\escpr}[1]{\langle#1\rangle}
\newcommand{\Sg}{\Sigma}

\newcommand{\Om}{\Omega}
\newcommand{\eps}{\varepsilon}

\newcommand{\ga}{\gamma}
\newcommand{\Ga}{\Gamma}

\newcommand{\nuh}{\nu_{h}}

\newcommand{\mnh}{|N_{h}|}
\newcommand{\nt}{\escpr{N,T}}

\newcommand{\n}{\nabla}

\newcommand{\ntnh}{\frac{\escpr{N,T}}{|N_h|}}

\usepackage{color}
\definecolor{grey}{rgb}{.7,.7,.7}

\DeclareMathOperator{\divv}{div}
\DeclareMathOperator{\tor}{Tor}

\DeclareMathOperator{\Jac}{Jac}
\DeclareMathOperator{\supp}{supp}

\newtheorem{theorem}{Theorem}[section]
\newtheorem{proposition}[theorem]{Proposition}
\newtheorem{lemma}[theorem]{Lemma}
\newtheorem{corollary}[theorem]{Corollary}

\theoremstyle{definition}

\theoremstyle{remark}

\newtheorem{remark}[theorem]{Remark}

\numberwithin{equation}{section}

\begin{document}

\title[Area-stationary and stable surfaces of class $C^1$ in $\hh^1$]{Area-stationary and stable surfaces of class $C^1$ \\ in the sub-Riemannian Heisenberg group $\hh^1$}

\author[M.~Galli]{Matteo Galli} \address{Dipartimento di Matematica \\
Universit\`a di Bologna \\ Piazza di Porta S. Donato 5 \\ 40126 Bologna, Italy}
\email{matteo.galli8@unibo.it}

\author[M.~Ritor\'e]{Manuel Ritor\'e} \address{Departamento de
Geometr\'{\i}a y Topolog\'{\i}a \\
Universidad de Granada \\ E--18071 Granada \\ Espa\~na}
\email{ritore@ugr.es}

\date{\today}

\thanks{M. Galli has been supported by the People Programme (Marie Curie Actions) of the European Union's Seventh
Framework Programme FP7/2007-2013/ under REA grant agreement n. 607643. M. Ritor\'e has been supported by Mec-Feder MTM2010-21206-C02-01 and Mineco-Feder MTM2013-48371-C2-1-P research grants.}

\begin{abstract}
We consider surfaces of class $C^1$ in the $3$-dimensional sub-Riemannian Heisenberg group $\hh^1$. Assuming the surface is area-stationary, i.e., a critical point of the sub-Riemannian perimeter under compactly supported variations, we show that its regular part is foliated by horizontal straight lines. In case the surface is complete and oriented, without singular points, and stable, i.e., a second order minimum of  perimeter, we prove that the surface must be a vertical plane. This implies the following Bernstein type result: a complete locally area-minimizing intrinsic graph of a $C^1$ function in $\hh^1$ is a vertical plane.
\end{abstract}

\maketitle

\thispagestyle{empty}

\bibliographystyle{abbrv} \nocite{*}

%\tableofcontents

\section{Introduction}
Variational problems related to the sub-Riemannian perimeter introduced by Capogna, Danielli and Garofalo \cite{MR1312686} (see also Garofalo and Nhieu \cite{MR1404326} and Franchi, Serapioni and Serra-Cassano \cite{MR1871966}) have received great interest recently, specially in the Heisenberg groups $\hh^n$. The recent monograph \cite{MR2312336} provides a quite complete survey of recent progress on the subject. Crucial open questions in this theory are the optimal regularity of minimizers, and their characterization under reasonable regularity assumptions.

The optimal regularity of perimeter minimizing sets in $\hh^n$ is still an open question. The boundaries of the conjectured solutions to the isoperimetric problem are of class $C^2$, see \cite{MR2312336}, although there exist examples of area-minimizing horizontal graphs which are just Euclidean Lipschitz, see \cite{MR2262784,MR2455341,MR2448649}. Cheng, Hwang and Yang \cite[Thms.~A, B]{MR2262784} proved existence and uniqueness of $t$-graphs which are Lipschitz continuous solutions of the minimal surface equation in $\hh^n$ under given Dirichlet conditions in a $p$-convex domain. Pinamonti, Serra Cassano, Treu and Vittone \cite[Thm.~1.1]{2013arXiv1307.4442P} have obtained existence and uniqueness of $t$-graphs on domains with boundary data satisfying a bounded slope condition. They also showed that, under this condition, Lipschitz regularity is optimal at least in the first Heisenberg group $\hh^1$, \cite[Example~6.5]{2013arXiv1307.4442P}. The bounded slope condition was also considered in a previous paper by Pauls \cite{MR2043961}. Capogna, Citti and Manfredini \cite{MR2583494} have proven that the intrinsic graph of a Lipschitz continuous function which is a vanishing viscosity solution of the sub-Riemannian minimal surface equation in $\hh^1$ is of class $C^{1,\alpha}$ and it is foliated by characteristic straight lines. A vanishing viscosity solution is nothing but a uniform limit of Riemannian minimal graphs uniformly bounded in Lipschitz norm. The same authors obtained higher regularity in $\hh^n$, $n>1$, in \cite{MR2774306}. 

Area-stationary surfaces of class $C^2$ in $\hh^1$ are well understood.  The regular set of such surfaces is known \cite{MR2165405, MR2435652} to be ruled by characteristic  segments. The singular set was completely described by Cheng, Hwang, Malchiodi and Yang \cite{MR2165405}, who proved that it is the union of isolated points and $C^1$ curves. The area-stationarity condition implies, in addition to have mean curvature zero, that the characteristic curves meet orthogonally the singular curves. This was proved for area-minimizing $t$-graphs by Cheng, Hwang, and Yang \cite{MR2262784}, and for general area-stationary surfaces by Ritor\'e and Rosales \cite{MR2435652}, who used this condition to describe the $C^2$ critical surfaces of the sub-Riemannian area in $\hh^1$ under a volume constraint.

Bernstein type problems for $C^2$ surfaces in $\hh^1$ have also received a special attention. In \cite{MR2165405}, a classification of all complete $C^2$ solutions to the minimal surface equation for $t$-graphs in $\hh^1$ was given. In \cite{MR2435652}, this classification was refined by showing that the only complete area-stationary $t$-graphs are Euclidean non-vertical planes or they are congruent to the hyperbolic paraboloid $t=xy$.  By means of a calibration argument it is also proved in \cite{MR2435652} that they are all area-minimizing. A classification of $C^2$ complete, connected, orientable, area-stationary surfaces with non-empty singular set was obtained in \cite{MR2435652}: the only examples are, modulo congruence, non-vertical Euclidean planes, the hyperbolic paraboloid $t=xy$, and the classical left-handed minimal helicoids. In \cite{MR2405158} and \cite{MR2333095} the Bernstein problem for \emph{intrinsic graphs} in $\hh^1$ was considered.  The notion of intrinsic graph is the one used by Franchi, Serapioni and Serra Cassano in \cite{MR1871966}.  Geometrically, an intrinsic graph is a normal graph over some Euclidean vertical plane with respect to the left invariant Riemannian metric $g$ in $\hh^1$.  A $C^1$ intrinsic graph has empty singular set and so the stationarity condition is equivalent to have the surface ruled by horizontal straight lines. Many examples of $C^2$ complete area-stationary intrinsic graphs different from vertical Euclidean planes were found in \cite{MR2405158}. A remarkable difference with respect to the case of $t$-graphs is the existence of complete $C^2$ area-stationary intrinsic graphs which are not area-minimizing, see \cite{MR2405158}.  In \cite{MR2333095}, Barone Adesi, Serra Cassano and Vittone classified complete $C^2$ area-stationary intrinsic graphs.  Then they computed the second variation formula of the area for such graphs to establish that the only stable ones, i.e., with non-negative second variation, are the Euclidean vertical planes. An interesting calibration argument, also given in \cite{MR2333095}, yields that the vertical planes are in fact area-minimizing surfaces in $\hh^1$.
Later on, using similar techniques, Danielli, Garofalo, Nhieu and Pauls \cite{MR2648078} proved that embedded stable complete surfaces without singular points were vertical planes. The following natural step was to consider complete stable surfaces in $\hh^1$. Finally in \cite{MR2609016}, Hurtado, Ritor\'e and Rosales proved that the only complete, orientable, connected, stable area-stationary surfaces in $\hh^1$ of class $C^2$ are the Euclidean planes and the surfaces congruent to the hyperbolic paraboloid $t=xy$. General first and second variation formulas of the sub-Riemannian area, moving the singular set, in pseudo-hermitian $3$-manifolds have been obtained by the first author in \cite{MR3044134}, also providing a classification \cite[Thm.~10.10]{MR3044134} of complete oriented $C^2$ stable surfaces in the roto-translational group, the group of rigid motions of the Euclidean plane. Classification results for $C^2$ complete oriented CMC stable surfaces without singular points in $3$-dimensional Sasakian sub-Riemannian manifolds have been obtained by Rosales \cite[Corollary~6.9]{MR2875642}. A CMC stable surface is a constant mean curvature surface with non-negative second variation of the area under a volume constraint. The first author has also considered the case of the sub-Riemannian Sol manifold, providing several results concerning complete $C^2$ area-stationary and stable surfaces \cite{MR3259763}.

%All the above results hold for surfaces of class $C^2$. However, there is evidence that this is not the natural regularity of minimizers in the Heisenberg.
In recent papers, Cheng, Hwang and Yang \cite{MR2481053} and Cheng, Hwang, Malchiodi and Yang \cite{MR2983199} have considered the case of $t$-graphs in $\hh^1$ of class $C^1$ which are weak solutions of the minimal surface equation. This condition turns out to be equivalent to be a critical point of the sub-Riemannian perimeter under variations by $t$-graphs.  In particular they showed that the regular set of such a graph is foliated by characteristic straight lines \cite[Thm.~A]{MR2481053}. In the minimal case, this result generalizes a previous one by Pauls \cite[Lemma~3.3]{MR2225631} for $H$-minimal surfaces with components of the horizontal Gauss map are in the class $W^{1,1}$. A quite complete description of the singular set of a $t$-graph of class $C^1$ which is a weak solution of the minimal surface equation was given in \cite[Thm.~C]{MR2983199}: it is composed of isolated points and continuous curves and it is locally path connected. Several curves may meet at a singular point. The key tool is the Codazzi type equation
\[
DD''=2\,(D'-1)(D'-2)
\]
satisfied by the sub-Riemannian area element $D=|N_h|/\escpr{N,T}$ along characteristic curves, see Remark~\ref{rem:codazzi}. Under non-degeneracy conditions, the regularity of the characteristic curves can be improved to $C^1$ and an equal angle condition for characteristic curves meeting the same singular (non-degenerate) point can be established, see \cite[Thms.~G, G']{MR2983199}.

%{\color{red} Cite also the paper by Bigolin and Serra-Cassano \cite{MR2600502}}

%Capogna, Citti and Manfredini \cite{MR2583494} proven that the intrinsic graph of a Lipschitz continuous function which is a vanishing viscosity solution of the sub-Riemannian minimal surface equation in $\hh^n$ is of class $C^{1,\alpha}$ and that it is foliated by characteristic geodesics. They obtained \cite{MR2774306} higher regularity $\hh^n$, $n>1$. A vanishing viscosity solution is nothing but a uniform limit of Riemannian minimal graphs uniformly bounded in Lipschitz norm.

%{\color{red} Look at recent papers}

In this paper we are interested in the Bernstein problem for intrinsic graphs of class $C^1$. First we need to extend Pauls and Cheng, Hwang and Yang characterization of characteristic curves to arbitrary area-stationary surfaces of class $C^1$. Throughout this paper, an area-stationary surface will be a critical point of the perimeter under any variation induced by a vector field in $\hh^1$ with compact support. After deriving the first variation formula for the sub-Riemannian area with a boundary term in Lemma~\ref{eq:gen1st}, we use Killing fields for the left-invariant Riemannian metric in $\hh^1$ (that preserve the sub-Riemannian perimeter) to obtain formula \eqref{eq:killing} in Lemma~\ref{eq:killing}. Formula \eqref{eq:killing} is similar to the flux formula obtained by Korevaar, Kusner and Solomon for constant mean curvature hypersurfaces in Euclidean space \cite{MR1010168}, and can be thought of as a special geometric version of the classical Noether's Theorem in Calculus of Variations. We remark that a balancing law for $t$-graphs of class $C^1$ was obtained by Cheng, Hwang and Yang \cite[Lemma~2.1]{MR2481053}, and that our arguments allow us to obtain theirs, see Remark~\ref{rem:chy}. The authors tried to use formula \eqref{eq:killing} to extend \cite[Thm.~A]{MR2481053} to arbitrary surfaces of class $C^1$ using different Killing vector fields. Unfortunately, the technique does not seem to work. With a different proof, localizing the first variation formula on characteristic curves, they shall prove in Theorem~\ref{thm:c1foliation}
\begin{quotation}
The regular part of a $C^1$ area-stationary surface $\Sg\subset\hh^1$ is foliated by horizontal straight lines.
\end{quotation}
This result has been extended by the authors to surfaces with prescribed mean curvature in \cite{1501.07246}.

Once we know about the structure of the regular set of an area-stationary surface of class $C^1$, we want to consider the second variation of these surfaces. Let us consider an oriented surface $\Sg\subset\hh^1$ of class $C^2$ with unit normal $N$ and horizontal unit normal $\nu_h$, and let $U$ be a horizontal vector field whose support does not contain singular points of $\Sg$. Then the second variation of the sub-Riemannian area can be computed as in \cite{MR2609016} to obtain
\begin{equation*}
\frac{d^2}{ds^2}\Big|_{s=0}A(\varphi_s(\Sg))=\int_\Sg \big\{Z(f)^2-qf^2\big\}\, (|N_h|\,d\Sg),
\end{equation*}
where $\{\varphi_t\}_{t\in\rr}$ is the flow generated by $U$, $Z$ is the characteristic vector field on $\Sg$, and $f=\escpr{U,\nu_h}$. The function $q$ is defined by 
\begin{equation*}
q:=4\,\bigg(Z\bigg(\frac{\escpr{N,T}}{|N_h|}\bigg)+\,\frac{\escpr{N,T}^2}{|N_h|^2}\bigg).
\end{equation*}
Hence a first task is to extend the second variation formula to a $C^1$ surface is to prove that the function $\escpr{N,T}/|N_h|$ is smooth when restricted to characteristic curves in $\Sg$. This is done by showing in Lemma~\ref{lem:parameterization}(3) that the function $\escpr{N,T}/|N_h|$ restricted to a characteristic straight line satisfies the ordinary differential equation \eqref{eq:codazzi}
\[
u''+6\,u'u+4\,u^3=0,
\]
where the primes indicate derivatives with respect to the arc-length parameter on the characteristic curve. This allows us to check the validity of the second variation formula \eqref{eq:2ndvar} in Theorem~\ref{thm:second}.  It is important to mention that equation~\eqref{eq:codazzi} is equivalent to the equation $DD''=2\,(D'-1)(D'-2)$ on $t$-graphs, see the discussion in Remark~\ref{rem:codazzi}.

Once the validity of the second variation formula \eqref{eq:2ndvar} is established, we define a stable surface as one satisfying inequality
\begin{equation*}
\int_\Sg \big\{Z(f)^2-qf^2\big\}\, (|N_h|\,d\Sg)\ge 0,
\end{equation*}
for any continuous function $f$ with compact support in the regular set of $\Sg$ and smooth in the horizontal direction in $\Sg$. Then we proceed as in the $C^2$ case to parameterize the surface using a seed curve $\Ga(\eps)$, orthogonal to the characteristic curves, and the orthogonal straight lines with vector $Z_{\Ga(\eps)}$. This way we produce a continuous parameterization
\[
(\eps,s)\mapsto \Ga(\eps)+s\,Z_{\Ga(\eps)},
\]
whose main properties are studied in Lemma~\ref{lem:parameterization}: it is smooth in the $s$-direction and almost everywhere smooth in the $\eps$-direction. Using this parameterization and the technical Lemma~\ref{lem:p} we are able to show in  Theorem~\ref{thm:main} 
\begin{quotation}
Let $\Sg\subset\hh^1$ be a complete oriented stable surface of class $C^1$ without singular points. Then $\Sg$ is a vertical plane.
\end{quotation}
As an immediate consequence we shall prove in  Corollary~\ref{cor:bernstein}
\begin{quotation}
Let $\Sg\subset\hh^1$ be a complete locally area-minimizing intrinsic graph of a $C^1$ function. Then $\Sg$ is a vertical plane.
\end{quotation}

We have organized this paper into five Sections. Some preliminary results, terminology and notation are in Section~\ref{sec:preliminaries}. The first variation formula and their consequences appear in Section~\ref{sec:first}, while the second variation is treated in Section~\ref{sec:2ndvariation}. The proof of the Stability Theorem and the Bernstein result are in Section~\ref{sec:second}.

\section{Preliminaries}
\label{sec:preliminaries}

In this section we gather some results to be used in later Sections

\subsection{The Heisenberg group}
\label{subsec:hg}
The \emph{Heisenberg group} $\hh^1$ is the Lie group $\rr^3\equiv\cc\times\rr$, with product
\[
[z,t]*[z',t']:=[z+z',t+t'+\text{Im}(z\overline{z}')], \qquad  [z,t], [z',t']\in\mathbb{C}\times\rr.
\]
Any $p\in\hh^1$ produces a \emph{left translation} defined $l_p(q)=p*q$.  A basis of left invariant vector fields is given by
\begin{equation*}
X:=\frac{\ptl}{\ptl x}+y\,\frac{\ptl}{\ptl t}, \qquad
Y:=\frac{\ptl}{\ptl y}-x\,\frac{\ptl}{\ptl t}, \qquad
T:=\frac{\ptl}{\ptl t}.
\end{equation*}
The \emph{horizontal distribution} $\mathcal{H}$ in $\hh^1$ is the smooth planar distribution generated by $X$ and $Y$.

We denote by $[U,V]$ the Lie bracket of two $C^1$ vector fields $U$ and $V$ on $\hh^1$.  Note that $[X,T]=[Y,T]=0$, while $[X,Y]=-2T$, so that $\mathcal{H}$ is a bracket-generating distribution.  Moreover,$\mathcal{H}$ is nonintegrable by Frobenius theorem.  The vector fields $X$ and $Y$ generate, at every point, the kernel of the (contact) $1$-form $\omega=-y\,dx+x\,dy+dt$.

\subsection{The left invariant metric}
\label{subsec:g}
We shall consider on $\hh^1$ the Riemannian metric $g=\escpr{\cdot\,,\cdot}$ so that $\{X,Y,T\}$ is an orthonormal basis at every point.  The restriction of $g$ to $\mathcal{H}$ coincides with the usual sub-Riemannian metric in $\hh^1$. The orthogonal \emph{horizontal projection} of a tangent vector $U$ onto $\mathcal{H}$ will be denoted by $U_{h}$.  A vector field $U$ is \emph{horizontal} if and only if $U=U_h$.

For any tangent vector $U$ on $\hh^1$ we define $J(U)=D_UT$, where $D$ is the Levi-Civita connection associated to the Riemannian metric $g$.  Then we have $J(X)=Y$, $J(Y)=-X$ and $J(T)=0$, so that $J^2=-\text{Id}$ when restricted to $\mathcal{H}$.  It is also clear 
\begin{equation}
\label{eq:conmute}
\escpr{J(U),V}+\escpr{U,J(V)}=0,
\end{equation}
for any pair of tangent vectors $U$ and $V$.  The involution $J:\mathcal{H}\to\mathcal{H}$ together with the $1$-form $\omega=-y\,dx+x\,dy+dt$, provides a pseudo-hermitian structure on $\hh^1$, see \cite[Sect.~6.4]{MR1874240}.

\subsection{The pseudo-hermitian connection}
\label{subsec:nabla}
In the Riemannian manifold $(\hh^1,g)$, we define the \emph{pseudo-hermitian connection} $\n$ as the unique metric connection, \cite[(I.5.3)]{MR2229062}, with torsion tensor given by
\begin{equation}
\label{def:torsiontensor}
\tor(X,Y):=2\escpr{J(X),Y}T. 
\end{equation}
For a more detailed introduction to $\n$ and its properties see \cite{MR2214654}. Here we only mention that $\n T\equiv 0$ and observe that the difference between $\n$ and the Levi-Civita connection $D$ can be computed using Koszul's formulas for $D$ and $\n$ to obtain
\[
D_XY-\n_XY=\escpr{Y,T}J(X)+\escpr{X,T}J(Y)-\escpr{J(X),Y}T.
\]

\subsection{Horizontal curves and Carnot-Carath\'eodory distance}
Let $\ga:I\to\hh^1$ be a piecewise $C^1$ curve defined on a compact interval $I\sub\rr$.  The \emph{length} of $\ga$ is the usual Riemannian length $L(\ga)=\int_{I}|\dot{\ga}(\eps)|\,d\eps$, where $\dot{\ga}$ is the tangent vector of $\ga$.  A \emph{horizontal curve} $\ga$ in $\hh^1$ is a $C^1$ curve whose tangent vector always lies in the horizontal distribution.  For two given points in $\hh^1$ we can find, by Chow's connectivity theorem \cite[Sect.~1.2.B]{MR1421823}, a horizontal curve joining these points.  The \emph{Carnot-Carath\'eodory distance} $d_{cc}$ between two points in $\hh^1$ is defined as the infimum of the length of horizontal curves joining the given points.  The topology associated to $d_{cc}$ coincides with the usual topology in $\rr^3$, see \cite[Cor.~2.6]{MR1421822}.

\subsection{Geodesics in $(\hh^1,g)$}
\label{subsec:geo}
A \emph{geodesic} in $(\hh^1,g)$ is a $C^2$ curve $\ga$ such that the covariant derivative of the tangent vector field $\dot{\ga}$ vanishes along $\ga$.

Let $\ga(s)=(x(s),y(s),t(s))$ be a horizontal geodesic. Dots will indicate derivatives with respect to $s$.  We write $\dot{\ga}=\dot{x}\,X+\dot{y}\,Y+(\dot{t}-\dot{x}y+x\dot{y})\,T$. Since $\ga$ is a horizontal curve,
\[
\dot{t}=\dot{x}y-x\dot{y}.
\]
Then $\ga$ is a horizontal geodesic in $(\hh^1,g)$ if and only if
\[
\ddot{x}=\ddot{y}=0.
\]
Solving these equations with initial conditions $(x(0),y(0),t(0))=(x_{0},y_{0},t_{0})$ and $(\dot{x}(0),\dot{y}(0))=(A,B)$ we get
\begin{equation}
\label{eq:eqgeo}
\begin{split}
x(s)&=x_{0}+A\,s,
\\
y(s)&=y_{0}+B\,s,
\\
t(s)&=t_{0}+(Ay_{0}-Bx_{0})\,s,
\end{split}
\end{equation}
%where $f$, $g$ and $h$ are the real analytic functions
%\[
%f(x)=
%\begin{cases}
%\displaystyle\frac{\sin(x)}{x}, \! &x\neq 0
%\\
%1, \! &x=0
%\end{cases},
%\quad
%g(x)=
%\begin{cases}
%\displaystyle\frac{1-\cos(x)}{x}, \! &x\neq 0
%\\
%0, \, &x=0
%\end{cases},
%\quad
%h(x)=
%\begin{cases}
%\displaystyle\frac{x-\sin(x)}{x^2}, \! &x\neq 0
%\\
%0, \, &x=0
%\end{cases}.
%\]
%In particular, we have
%\begin{equation}
%\label{eq:horgeo}
%\exp_{p}(sv)=p+sv,\qquad\text{for } \, p\in\hh^1 \ \text{and }
%\, v\in\mathcal{H}_{p},
%\end{equation}
which is a horizontal straight line.% Here $\exp_{p}$ denotes the exponential map of $(\hh^1,g)$ at~$p$.

\subsection{Geometry of surfaces in $\hh^1$}
\label{subsec:surfaces}
Let $\Sg$ be a $C^1$ surface immersed in $\hh^1$.  The \emph{singular set} $\Sg_0$ consists of those points $p\in\Sg$ for which the tangent plane $T_p\Sg$ coincides with $\mathcal{H}_{p}$.  As $\Sg_0$ is closed and has empty interior in $\Sg$, the \emph{regular set} $\Sg\setminus\Sg_0$ of $\Sg$ is open and dense in $\Sg$.  It was proved in \cite[Lemma~1]{MR0343011}, see also \cite[Thm.~1.2]{MR2021034}, that, for a $C^2$ surface, the Hausdorff dimension of $\Sg_{0}$ with respect to the Riemannian distance on $\hh^1$ is less than or equal to one.  In particular, the Riemannian area of $\Sg_{0}$ vanishes.  If $N$ is a unit normal vector to $\Sg$ in $(\hh^1,g)$, then we can describe the singular set as $\Sg_{0}=\{p\in\Sg;N_h(p)=0\}$, where $N_{h}=N-\escpr{N,T}T$.  In the regular part $\Sg\setminus\Sg_0$, we can define the \emph{horizontal Gauss map} $\nu_h$ and the \emph{characteristic vector field} $Z$, by
\begin{equation}
\label{eq:nuh}
\nu_h:=\frac{N_h}{|N_h|}, \qquad Z:=J(\nuh).
\end{equation}
As $Z$ is horizontal and orthogonal to $\nu_h$, we conclude that $Z$ is tangent to $\Sg$.  Hence $Z_{p}$ generates $T_{p}\Sg\cap\mathcal{H}_{p}$. The integral curves of $Z$ in $\Sg\setminus\Sg_0$ will be called $(\!\emph{oriented})$ \emph{characteristic curves} of $\Sg$.  They are both tangent to $\Sg$ and horizontal.  If we define
\begin{equation}
\label{eq:ese}
S:=\escpr{N,T}\,\nu_h-|N_h|\,T,
\end{equation}
then $\{Z_{p},S_{p}\}$ is an orthonormal basis of $T_p\Sg$ whenever $p\in\Sg\setminus\Sg_0$.  Moreover, for any $p\in\Sg\setminus\Sg_{0}$ we have the orthonormal basis of $T_{p}\hh^1$ given by $\{Z_{p},(\nuh)_{p},T_{p} \}$.  From here we deduce the following identities on $\Sg\setminus\Sg_{0}$
\begin{equation}
\label{eq:relations}
|N_{h}|^2+\escpr{N,T}^2=1, \qquad (\nu_{h})^\top=\escpr{N,T}\,S, \qquad
T^\top=-|N_{h}|\,S,
\end{equation}
where $U^\top$ stands for the projection of a vector field $U$ onto
the tangent plane to $\Sg$. A \emph{seed} curve in $\Sg\setminus\Sg_0$ is an integral curve of the vector field $S$. Since $S$ is continuous, existence of seed curves is guarenteed, but not their uniqueness.

Given a $C^1$ immersed surface $\Sg$ with a unit normal vector $N$, we define the \emph{sub-Riemannian area} of $\Sg$ by
\begin{equation}
\label{eq:area}
A(\Sg):=\int_{\Sg}|N_{h}|\,d\Sg,
\end{equation}
where $d\Sg$ is the Riemannian area element on $\Sg$.  If $\Sg$ is a $C^2$ surface bounding a set $\Om$, then $A(\Sg)$ coincides with all the notions of perimeter of $\Om$ and area of $\Sg$ introduced by other authors, see \cite[Prop.~2.14]{MR1871966}, \cite[Thm.~5.1]{MR1865002} and \cite[Cor.~7.7]{MR1871966}.

We shall say that an immersed surface $\Sg\subset\hh^1$ is \emph{complete} if it is complete in $(\hh^1,g)$. The \emph{$t$-graph} of a function $u:\Om\subset \{t=0\}\to\rr$ is the set $\{(x,y,u(x,y)): (x,y)\in\Om\}$. Given a vertical plane $\Pi$, the \emph{intrinsic graph} of the function $u:\Om\subset\Pi\to\rr$ is the Riemannian graph of $u$ over $\Om\subset\Pi$, see \cite{MR2223801}. This means the set $\{\exp_p(u(p)\,N_p):p\in\Om\}$, where $\exp_p$ is the Riemannian exponential map  in $(\hh^1,g)$ at $p$, and $N_p$ is a given unit normal to $\Pi$ with respect to the Riemannian metric $g$.

Given a vector field $U$ in $\hh^1$, its \emph{horizontal tangent projection} is
\[
U_{ht}:=U-\escpr{U,\nu_h}\,\nu_h-\escpr{U,T}\,T.
\]
It is clear that $U_{ht}=\escpr{Z,U}\,Z$.

Given a vector field $U$ in $\hh^1$, we define its \emph{horizontal divergence} on $\Sg\setminus\Sg_0$ by
\begin{equation}
\label{eq:hordiv}
\divv_\Sg^h U:=\escpr{\nabla_Z U,Z},
\end{equation}
where $\nabla$ is the pseudo-hermitian connection. The horizontal divergence can be defined for any vector field smooth along horizontal curves in $\Sg$.

\section{First variation and area-stationary surfaces}
\label{sec:first}

Let $\Sg\subset\hh^1$ be a $C^1$ surface. We shall say $\Sg$ is \emph{area-stationary} if, for any $C^1$ vector field $U$ in $\hh^1$ with compact support such that $\text{supp}(U)\cap\ptl\Sg=\emptyset$, and associated one-parameter group of diffeomorphisms $\{\varphi_s\}_{s\in\rr}$, we have
\[
\frac{d}{ds}\bigg|_{s=0} A(\varphi_s(\Sg))=0,
\]
where $A$ is the sub-Riemannian area in $\hh^1$. As shown in the first result, the first variation is computable for any $C^1$ surface $\Sg$

\begin{lemma}
\label{lem:1st}
Let $\Sg\subset\hh^1$ be a $C^1$ surface, possibly with non-empty boundary, and $U\in C^1$ a vector field with compact support and associated one-parameter group of diffeomorphisms $\{\varphi_s\}_{s\in\rr}$. Then
\begin{equation}
\label{eq:1st-0}
\frac{d}{ds}\bigg|_{s=0} A(\varphi_s(\Sg))=\int_\Sg\big\{-S\escpr{U,T}-2\,\escpr{J(U),S}+|N_h|\,\divv_\Sg^h U\big\}\,d\Sg.
\end{equation}
%Moreover, if $\Sg$ is at least of class $C^2$ we have, in absence of singular points
%\begin{equation}
%\label{eq:1st-c2}
%\frac{d}{ds}\bigg|_{s=0} A(\varphi_s(\Sg))=\int_\Sg\big\{-\divv_\Sg\big(\escpr{U,T}\,S\big)+\divv_\Sg\big(|N_h|\,U_{ht}\big)+H\,\escpr{U,N}\big\}\,d\Sg,
%\end{equation}
%where $\divv_\Sg$ is the usual standard Riemannian divergence on $\Sg$ and $H$ is the mean curvature of $\Sg$. Furthermore, in the $C^2$ case, the Riemannian Divergence Theorem implies
%\begin{equation}
%\label{eq:1st-c2bdy}
%\frac{d}{ds}\bigg|_{s=0} A(\varphi_s(\Sg))=\int_{\ptl\Sg}\escpr{U,T}\escpr{S,\xi}-\int_{\ptl\Sg} |N_h|\,\escpr{U_{ht},\xi}+\int_\Sg H\,\escpr{U,N}\,d\Sg,
%\end{equation}
%where $\xi$ is the unit inner conormal to $\ptl\Sg$.
\end{lemma}

We remark that formula \eqref{eq:1st-0} %, \eqref{eq:1st-c2}, \eqref{eq:1st-c2bdy} are 
is valid for any variation, even for those moving the boundary of $\Sg$.

\begin{proof}%{Proof of Lemma~\ref{lem:1st}}
Equation \eqref{eq:1st-0} is exactly \cite[Lemma~3.4]{MR3044134} for the case of the Heisenberg group~$\hh^1$.%as well as \eqref{eq:1st-c2} follows from \cite[Corollary~3.5 and Corollary~3.6]{MR3044134}. Finally \eqref{eq:1st-c2bdy} is obtained applying the Riemannian Divergence theorem to \eqref{eq:1st-c2}.
%The first equation \eqref{eq:1st-0} is exactly \cite[Lemma~3.4]{MR3044134} restricted to the case of the Heisenberg group $\hh^1$ as well as \eqref{eq:1st-c2} follows from \cite[Corollary~3.5 and Corollary~3.6]{MR3044134}. Finally \eqref{eq:1st-c2bdy} is obtained applying the Riemannian Divergence theorem to \eqref{eq:1st-c2}.
\end{proof}

The following result will be very useful to obtain geometric properties of an area-stationary surface

\begin{lemma}
\label{eq:gen1st}
Let $\Sg\subset\hh^1$ be an area-stationary surface of class $C^1$ without singular points, and let $U$ any vector field in $\hh^1$ of class $C^1$. Then, for any  subset $\Sg'$ compactly contained in $\Sg$ with piecewise $C^1$ boundary $\ptl\Sg'$,
\begin{equation}
\label{eq:1stweak}
0=\int_{\ptl\Sg'}\escpr{\escpr{U,T}\,S-|N_h|\,U_{ht},\xi}\,d(\ptl\Sg')+
\int_\Sg\big\{-S\escpr{U,T}-2\,\escpr{J(U),S}+|N_h|\,\divv_\Sg^h U\big\}\,d\Sg.
\end{equation}
where $\xi$ is the unit inner conormal to $\ptl\Sg'$.
\end{lemma}

\begin{proof}
Let $U$ be a $C^1$ vector field in $\hh^1$, and $f$ a smooth function in $\hh^1$ with compact support, chosen so that $\text{supp}(f)\cap\ptl\Sg=\emptyset$. Since $\Sg$ is area-stationary, the first variation of the sub-Riemannian area with respect to the flow associated to $fU$ is zero. From \eqref{eq:1st-0} this means
\[
\int_\Sg\{-S\escpr{fU,T}-2\,\escpr{J(fU),S}+|N_h|\,\divv_\Sg^h (fU)\}\,d\Sg=0.
\]
This formula is also valid by approximation for Lipschitz functions on $\Sg$, and can be rewritten as
\begin{equation}
\label{eq:f}
0=\int_\Sg\escpr{W,\nabla_\Sg f}\,d\Sg+\int_\Sg fg\,d\Sg,
\end{equation}
where $W$ is the tangent vector field $-\escpr{U,T}\,S+|N_h|\,\escpr{U,Z}\,Z=-\escpr{U,T}\,S+|N_h|\,U_{ht}$, and $g$ is the continuous function $-S\escpr{U,T}-2\,\escpr{J(U),S}+|N_h|\,\divv_\Sg^h U$.

Take now $\Sg'\subset\Sg$ relatively compact in $\Sg$ with $C^1$ boundary. Consider the sequence of lipschitz functions $f_\eps$ on $\Sg$ so that $f_\eps$ equals $1$ on $\Sg'$, vanishes when the distance to $\Sg'$ is larger than of equal to $\eps$ and is a linear function of the distance to $\Sg'$ in between. A standard argument using the coarea formula for Lipschitz functions shows
\begin{equation}
\label{eq:weakdiv}
\lim_{\eps\to 0}\int_\Sg \escpr{W,\nabla_\Sg f_\eps}\,d\Sg=\int_{\ptl\Sg'} \escpr{W,\xi}\,d(\ptl\Sg'),
\end{equation}
where $\xi$ is the inner unit normal to $\ptl\Sg'$. As the function $g$ is continuous in $\Sg$ and the functions $f_\eps$ are bounded, we can use the Dominated Convergence Theorem to show
\begin{equation}
\label{eq:limitsg'}
\lim_{\eps\to 0}\int_\Sg f_\eps\,g\,d\Sg=\int_{\Sg'} g\,d\Sg,
\end{equation}
since $f_\eps$ approaches the characteristic function of the set $\Sg'$. Taking limits, as $\eps\to 0$, in formula \eqref{eq:f} applied to $f=f_\eps$, and using \eqref{eq:weakdiv} and \eqref{eq:limitsg'} together with the first variation formula \eqref{eq:1st-0}, we obtain \eqref{eq:1stweak}.
\end{proof}

Recall that right-invariant vector fields in $\hh^1$ are Killing fields for the left-invariant Riemannian metric. They are linear combination, with constant coefficients, of the vector fields $X-2yT$, $Y+2xT$ and $T$. They induce one-parameter groups of isometries, for the left-invariant Riemannian metric in $\hh^1$, and are infinitesimal generators for one-parameter groups of contact transformations, see Liberman \cite{MR0119153} and Korányi-Reimann \cite[\S~5.1]{MR1317384}. Hence the associated one-parameter groups of diffeomorphisms preserve the sub-Riemannian perimeter, and we obtain

\begin{lemma}
\label{lem:killing}
Let $\Sg\subset\hh^1$ be a $C^1$ area-stationary surface without singular points. If $U$ is a Killing field and $\Sg'$ is compactly contained in $\Sg$ with $C^1$ boundary $\ptl\Sg'$, then
\begin{equation}
\label{eq:killing}
\int_{\ptl\Sg'}\escpr{\escpr{U,T}\,S-|N_h|\,U_{ht},\xi}=0,
\end{equation}
where $\xi$ is the unit inner conormal to $\ptl\Sg'$.
\end{lemma}

%In particular, taking $U=T$ we get
%\[
%\int_{\ptl\Sg'}\escpr{S,\xi}=0.
%\]
%Taking $U=Y+2xT$ we obtain
%\[
%\int_{\ptl\Sg'}\escpr{2xS-|N_h|U_{ht},\xi}=0.
%\]

\begin{remark}
\label{rem:chy}
The main consequence of Lemma~\ref{lem:killing} is that the regular part of an area-stationary $t$-graph $\Sg$ is foliated by horizontal straight lines. This result was first proved by Cheng, Hwang and Yang \cite{MR2481053}, see Remark~\ref{rem:chy}. We can use \eqref{eq:killing} to recover their result in the following way: taking as Killing field $U=T$ in Lemma~\ref{lem:killing}, formula \eqref{eq:killing} reads
\[
\int_{\ptl\Sg'}\escpr{S,\xi}=0.
\]
Assume $\Sg$ is the graph of the $C^1$ function $u:\Om\subset\rr^2\to\rr$. Simple computations allow us to write the above formula as
\begin{equation}
\label{eq:divt}
\int_{\ptl\Om'} \escpr{V,\xi_0}\,d(\ptl\Om')=0,
\end{equation}
where $V=(\nabla u+F)/|\nabla u+F|$, $\nabla$ is the gradient in $\rr^2$, $F(x,y)=(-y,x)$, $\Om'\subset\Om$ is relatively compact, and $\xi_0$ is the Euclidean unit normal to $\ptl\Om'$. Using that the Euclidean modulus of $V$ is equal to $1$, the classical Divergence Theorem, equality \eqref{eq:divt}, and a calibration argument, we deduce that the trajectories of the vector field $V^\perp$ locally minimize length. Hence they are Euclidean straight lines, and their horizontal liftings are horizontal geodesics.
\end{remark}

The authors have tried to use formula \eqref{eq:killing} to extend Theorem~A in \cite{MR2481053} to arbitrary surfaces of class $C^1$ using different Killing vector fields. Unfortunately, the technique does not seem to work. A proof using Calculus of Variations arguments is given below.

\begin{theorem}
\label{thm:c1foliation}
The regular part of a $C^1$ area-stationary surface $\Sg\subset\hh^1$ is foliated by horizontal straight lines.
\end{theorem}

\begin{proof}
Given a point $p\in\Sg\setminus\Sg_0$, there exists a vertical plane $\Pi$ through the origin so that an open neighborhood $B\subset\Sg$ of $p$ is an intrinsic graph over an open subset $\Om\subset\Pi$. Rotating $\Sg$ around the $t$-axis we may assume that $\Pi$ is the plane $y=0$, and that $B$ is the vertical graph $G_u$ of a function $u:\Om\to\rr$ of class $C^1$. This implies
\[
G_u=\{(x,u(x,t),t-xu(x,t)): (x,t)\in\Om\}.
\]
The tangent plane to any point in $G_u$ is generated by the vectors
\begin{align*}
\tfrac{\ptl}{\ptl x}&\mapsto (1,u_x,-u-xu_x)=X+u_xY-2uT,
\\
\tfrac{\ptl}{\ptl t}&\mapsto (0,u_t,1-xu_t)=u_tY+T.
\end{align*}
The ``downward'' pointing unit normal vector is then given by $\widetilde{N}/|\widetilde{N}|$, where
\[
\widetilde{N}=(X+u_xY-2uT)\times (u_tY+T)
=(u_x+2uu_t)\,X-Y+u_tT,
\]
and $\times$ denotes the cross product with respect to the oriented orthonormal basis $\{X,Y,T\}$. The horizontal unit normal is given by $\nuh=\widetilde{\nu}_h/|\widetilde{\nu}_h|$, where
\[
\widetilde{\nu}_h=(u_x+2uu_t)\,X-Y,
\]
and the characteristic direction by $Z=\widetilde{Z}/|\widetilde{Z}|$, where
\[
\widetilde{Z}=J(\widetilde{\nu}_h)=X+(u_x+2uu_t)Y.
\]

Assume now that $\ga(s)=(x(s),t(s))$ is a curve in $\Om$. Then
\[
\Ga(s)=(x(s),u(x(s),t(s)),t(s)-x(s)u(x(s),t(s)))\subset G_u,
\]
and we get
\[
\Ga'(s)=x'\,(X+u_xY-2uT)+t'\,(u_tY+T)=x'X+(x'u_x+t'u_t)Y+(t'-2ux')T.
\]
In particular, horizontal curves in $G_u$ satisfy the ordinary differential equation $t'=2ux'$. Since $u\in C^1(\Om)$, we have uniqueness of characteristic curves through any given point in $G_u$.

The jacobian of the map $(x,t)\mapsto (x,u(x,t),t-xu(x,t))$ can be easily computed and allows us to express the Riemannian area element $d\Sg$ in terms of the Lebesgue measure $dx\,dt$ of the vertical plane
\[
d\Sg=(1+u_t^2+(u_x+2uu_t)^2)^{1/2}dxdt=|\widetilde{N}|\,dxdt.
\]
On the other hand, since 
\[
|N_h|=\frac{(1+(u_x+2uu_t)^2)^{1/2}}{|\widetilde{N}|},
\]
we obtain the following expression for the sub-Riemannian area of the intrinsic graph of the $C^1$ function $u:\Om\to\rr$
\begin{equation}
\label{eq:areagraph}
A(G_u)=\int_\Om (1+(u_x+2uu_t)^2)^{1/2}dx\,dt.
\end{equation}

%From now on we assume that $\Om$ is a rectangular domain, the product of two real intervals.
Choose a $C^1$ function $\varphi$ with compact support in $\Om$ and consider the variation of $G_u$ by the intrinsic graphs $u+s\varphi$. This variation corresponds to moving the graph $G_u$ under the action of the vector field $\varphi Y$, where $\varphi$ is considered as a function in $\hh^1$ by means of the intrinsic orthogonal projection on the plane $\Pi$. The derivative of the sub-Riemannian area is given by
\begin{equation}
\label{eq:firstgraph}
0=\frac{d}{ds}\bigg|_{s=0} A(G_{u+s\varphi})
%&=\frac{d}{ds}\bigg|_{s=0}\int_\Om (1+\{(u+s\varphi)_x+2\,(u+s\varphi)(u+s\varphi)_t\}^2)^{1/2}dx\,dt
%\\ &
=\int_\Om \frac{u_x+2uu_t}{(1+(u_x+2uu_t)^2)^{1/2}}\,(\varphi_x+2u\varphi_t+2\varphi u_t)\,dx\,dt.
\end{equation}
%Since the vector field $U$ is given by $\varphi Y$, the integral $\int_\Sg H\,\escpr{U,N}\,d\Sg$ is given by
%\begin{equation}
%\label{eq:derinth}
%\int_\Sg H\,\escpr{U,N}\,d\Sg=-\int_\Om H\varphi\,\frac{1}{|\widetilde{N}|}\, |\widetilde{N}|\,dx\,dt=-\int_\Om H\varphi\,dx\,dt.
%\end{equation}

Assume the point $p\in G_u$ corresponds to the point $(a,b)$ in the $xt$-plane. The  curve $s\mapsto (s,t(s))$ is (a reparameterization of the projection of) a characteristic curve if and only if the function $t(s)$ satisfies the ordinary differential equation $t'(s)=u(s,t(s))$. For $\eps$ small enough, we consider the solution $t_\eps$ of equation $t_\eps'(s)=2u(s,t_\eps(s))$ with initial condition $t_\eps(a)=b+\eps$, and define $\ga_\eps(s):=(s,t_\eps(s))$, with $\ga=\ga_0$. We may assume that, for small enough $\eps$, the functions $t_\eps$ are defined in the interval $[a-r,a+r]$ for some $r>0$. By Peano's Theorem \cite[Thm.~3.1]{MR0171038}, $t_\eps(s)$ is of class $C^1$ as a function of $\eps$ and $s$. Moreover, the function $\ptl t_\eps/\ptl\eps$ satisfies
\begin{equation}
\label{eq:dteps}
\bigg(\frac{\ptl t_\eps}{\ptl\eps}\bigg)'(s)=2u_t(s,t_\eps(s))\,\bigg(\frac{\ptl t_\eps}{\ptl\eps}\bigg)(s),\qquad \frac{\ptl t_\eps}{\ptl\eps}(a)=1.
\end{equation}
where $'$ is the derivative with respect to the parameter $s$.

We consider the $C^1$ parameterization
\[
F(\xi,\eps):= (\xi,t_\eps(\xi))=(s,t)
\]
near the characteristic curve through $(a,b)$. The jacobian of this parameterization is given by
\[
\bigg|\begin{pmatrix}
1 & t'_\eps
\\
0 & \ptl t_\eps/\ptl \eps
\end{pmatrix}\bigg|=\frac{\ptl t_\eps}{\ptl\eps},
\]
which is positive because of the choice of initial condition for $t_\eps$ and the fact that the curves $\ga_\eps(s)$ foliate a neighborhood of $(a,b)$.  This Jacobian is also differentiable in direction of $\xi$ by \eqref{eq:dteps}. Any function $\varphi$ can be considered as a function of the variables $(\xi,\eps)$ by making $\tilde{\varphi}(\xi,\eps):=\varphi(\xi,t_\eps(\xi))$. Changing variables, and assuming the support of $\varphi$ is contained in a sufficiently small neighborhood of $(a,b)$, we can express the integral \eqref{eq:areagraph} as
\[
\int_{I}\bigg\{\int_{a-r}^{a+r}\frac{\ptl\tilde{u}/\ptl\xi}{(1+(\ptl\tilde{u}/\ptl\xi)^2)^{1/2}}\,\bigg(\frac{\ptl\tilde{\varphi}}{\ptl\xi}+2\tilde{\varphi} \tilde{u}_t\bigg)\,\frac{\ptl t_\eps}{\ptl\eps}\,d\xi\bigg\}\,d\eps,
\]
where $I$ is a small interval containing $0$. Instead of $\tilde{\varphi}$, we can consider the function $\tilde{\varphi} h /( t_{\eps+h}-t_\eps)$, where $h$ is a sufficiently small real parameter. We get that 
\[
\frac{ \ptl}{\ptl\xi}\bigg(\frac{h\, \tilde{\varphi}}{ t_{\eps+h}-t_\eps}\bigg)=\frac{\ptl\tilde{\varphi}}{\ptl\xi}\cdot \frac{h}{t_{\eps+h}-t_\eps}-2\tilde{\varphi}\cdot \frac{\tilde{u}(\xi, \eps+h)-\tilde{u}(\xi,\eps)}{t_{\eps+h}-t_\eps}\cdot \frac{h}{t_{\eps+h}-t_\eps}
\]
tends to 
\[
\frac{\ptl\tilde{\varphi}/\ptl\xi}{\ptl t_\eps/\ptl\eps}-\frac{2\tilde{\varphi}\tilde{u}_t}{\ptl t_\eps/\ptl\eps},
\]
when $h\rightarrow 0$. So using that $G_u$ is area-stationary we have that
\[
\int_{I}\bigg\{\int_{a-r}^{a+r} \frac{h}{t_{\eps+h}-t_\eps}   \frac{\ptl\tilde{u}/\ptl\xi}{(1+(\ptl\tilde{u}/\ptl\xi)^2)^{1/2}}  \,\bigg(\frac{\ptl\tilde{\varphi}}{\ptl\xi}\cdot +2  \tilde{\varphi}\cdot \bigg(\tilde{u}_t-\frac{\tilde{u}(\xi, \eps+h)-\tilde{u}(\xi,\eps)}{t_{\eps+h}-t_\eps}\bigg)\bigg)\,\frac{\ptl t_\eps}{\ptl\eps}\,d\xi\bigg\}\,d\eps
\]
vanishes. Furthermore, letting $h\rightarrow 0$ we conclude
\[
\int_{I}\bigg\{\int_{a-r}^{a+r}\bigg(\frac{\ptl\tilde{u}/\ptl\xi}{(1+(\ptl\tilde{u}/\ptl\xi)^2)^{1/2}}  \,\frac{\ptl\tilde{\varphi}}{\ptl\xi}\bigg)\,d\xi\bigg\}\,d\eps=0.
\]

% Instead of $\tilde{\varphi}$, we can consider the function $\tilde{\varphi}/(\ptl t_\eps/\ptl\eps)$. Using \eqref{eq:dteps} we get
%\[
%\frac{\ptl}{\ptl\xi}\bigg(\frac{\tilde{\varphi}}{\ptl t_\eps/\ptl\eps}\bigg)=\frac{\ptl\tilde{\varphi}/\ptl\xi}{\ptl t_\eps/\ptl\eps}-\frac{2\tilde{\varphi}\tilde{u_t}}{\ptl t_\eps/\ptl\eps},
%\]
%and so the first variation for a function of the form $\tilde{\varphi}/(\ptl t_\eps/\ptl\eps)$ is given by
%\[
%\int_{I}\bigg\{\int_{a-r}^{a+r}\frac{\ptl\tilde{u}/\ptl\xi}{(1+(\ptl\tilde{u}/\ptl\xi)^2)^{1/2}}\,\frac{\ptl\tilde{\varphi}}{\ptl\xi}\,d\xi\bigg\}\,d\eps.
%\]

Let $\eta:\rr\to\rr$ be a positive function with compact support in the interval $I$ and consider the family $\eta_\rho(x):=\rho^{-1}\eta(x/\rho)$. Inserting a test function of the form $\eta_\rho(\eps)\psi(\xi)$, where $\psi$ is a $C^\infty$ function with compact support in $(a-r,a+r)$, making $\rho\to 0$, and using that $G_u$ is area-stationary we obtain
\[
\int_{a-r}^{a+r} \frac{\ptl\tilde{u}/\ptl\xi}{(1+(\ptl\tilde{u}/\ptl\xi)^2)^{1/2}}(0,\xi)\,\psi'(\xi)\,d\xi=0
\]
for any $\psi\in C^\infty_0((a-r,a+r))$. This implies that the quantity
\[
\frac{\ptl\tilde{u}/\ptl\xi}{(1+(\ptl\tilde{u}/\ptl\xi)^2)^{1/2}}(0,\xi)=\frac{u_x+2uu_t}{(1+(u_x+2uu_t)^2)^{1/2}}(s,t(s))
\]
is constant along the curve $\ga(s)$, and so it is $(u_x+2uu_t)$. Hence $Z$ is the restriction of a left-invariant horizontal vector field to the characteristic curve passing through $p$. This implies that the characteristic curve is a horizontal straight line.
\end{proof}

%\begin{lemma}
%\label{lem:elementary}
%Let $J\subset\rr$ be an open interval, and $h:J\to\rr$ a continuous function. Assume
%\[
%\int_a^b\bigg(s-\frac{a+b}{2}\bigg)\,h(s)\,ds=0, \qquad \forall\ a,b\in J,\ a<b.
%\]
%Then $h$ is a constant function on $J$.
%\end{lemma}
%
%\begin{proof}
%Consider the function
%\[
%F(a,b)=\int_a^b\bigg(s-\frac{a+b}{2}\bigg)\,h(s)\,ds,
%\]
%which is of class $C^1$ in the open set $\{(a,b): a,b\in J, a<b\}$. Since
%\begin{align*}
%0=\frac{\ptl F}{\ptl a}&=-\bigg(\frac{a-b}{2}\bigg)\,h(a)-\frac{1}{2}\int_a^bh(s)\,ds,
%\\
%0=\frac{\ptl F}{\ptl b}&=\bigg(\frac{b-a}{2}\bigg)\,h(b)-\frac{1}{2}\int_a^bh(s)\,ds,
%\end{align*}
%we get $h(a)=h(b)$ for any $a$, $b\in J$, $a<b$.
%\end{proof}

%The following observation will be useful in the next Sections

\begin{remark}
\label{rem:Z}
From the proof of Theorem~\ref{thm:c1foliation} it is easy to check that there exists a \emph{unique} characteristic curve in a $C^1$ surface $\Sg\subset\hh^1$ passing through any point $p\in\Sg\setminus\Sg_0$, a result already proved in \cite[Thm.~B']{MR2481053}.

Indeed, using a rotation around the $t$-axis we may assume that, near $p$, the surface $\Sg$ is the intrinsic graph of a $C^1$ function $u$ defined on an open set $\Om$ of the vertical plane $\{y=0\}$. The computations in the first part of the proof of Theorem~\ref{thm:c1foliation} imply that (a reparametrization of) the projection of the characteristic curves to $\Om$ satisfy the ordinary differential equation $t'=2ux'$. This is equivalent to the system $t'=2u, x'=1$, whose solutions are integral curves of the $C^1$ vector field
\[
W=\frac{\ptl}{\ptl x}+2u\,\frac{\ptl}{\ptl t}.
\]
This implies local uniqueness of the projection of the integral curves of $Z$ and so local uniqueness of such curves. Global uniqueness follows from an standard connectedness argument.
\end{remark}

If $\Sg\subset\hh^1$ is any immersed oriented $C^1$ surface such that the vectors $Z$ and $\nu_h$ are $C^\infty$ along characteristic curves in the regular part of $\Sg$, we define the \emph{mean curvature} of $\Sg$ at any point $p\in\Sg\setminus\Sg_0$, by
\begin{equation}\label{def:meancurvature}
H(p):=(\divv_\Sg^h\nu_h)(p),
\end{equation}
where $\divv_\Sg^h(\nu_h)(p)$ is defined by $\escpr{\n_Z \nu_h,Z}$ as in \eqref{eq:hordiv}. Theorem \ref{thm:c1foliation} implies that, in an area-stationary surface, $Z$ and $\nu_h$ are $C^\infty$ along characteristic curves in the regular part of $\Sg$.

Under these regularity hypothesis, we can be able to perform the following first variation formula,

\begin{proposition}\label{1var:weakstationary} Let $\Sg\subset\hh^1$ be an oriented ${C}^1$ surface. Assume $\nu_h$ and $Z$ are differentiable in the $Z$-direction. Then the first variation of the area induced by a $C^1$ vector field $U$ such that $\supp(U)\cap\Sg\subset \Sg\setminus\Sg_0$ is given by
\begin{equation*}
\frac{d}{ds}\Big{|}_{s=0} A(\varphi_{s}(\Sg))=-\int_{\Sg} \escpr{U,N}\,H\,d\Sg,
\end{equation*}
where $H$ is defined in \eqref{def:meancurvature}.
\end{proposition}

\begin{proof} We decompose $U$ as $U=\escpr{U,\nu_h}\nu_h+\escpr{U,Z}Z+\escpr{U,T}T$, and observe that we can split the first variation \eqref{eq:1st-0} as
\begin{equation*}
\begin{split}
\frac{d}{ds}\Big{|}_{s=0} A(\varphi_{s}(\Sg))=&-\int_\Sg S(\escpr{U,T}) \,  d\Sg    + \int_\Sg   \mnh \divv_\Sg^h (\escpr{U,\nu_h}\nu_h)\,d\Sg\\
& +\int_\Sg \{  -2\escpr{U,Z}\escpr{J(Z),S}+\mnh \divv_\Sg^h (\escpr{U,Z}Z) \}\, d\Sg. 
\end{split}
\end{equation*}
We observe that all terms in the previous equation are well defined, as $U\in C^1$ and $Z$ and $\nu_h$ are $C^1$ in the $Z$-direction. 
We claim that the following integral equalities hold
\begin{align*}
 \int_\Sg   \mnh \divv_\Sg^h (\escpr{U,\nu_h}\nu_h)\,d\Sg&=\int_{\Sg}|N_{h}|\escpr{U,\nu_h}\escpr{\n_{Z}\nu_h,Z} d\Sg,
\\
\int_{\Sg}S(\escpr{U,T})\,d\Sg&=\int_{\Sg}g(N,T)\escpr{U,T}\escpr{\n_{Z}\nu_h,Z} d\Sg
\\
\int_\Sg \{  -2\escpr{U,Z}\escpr{J(Z),S}&+\mnh \divv_\Sg^h (\escpr{U,Z}Z) \}\, d\Sg=0.
\end{align*}
Indeed we can argue as in the proof of \cite[Proposition 6.3]{MR3044134} since, by \cite[Remark 6.1]{MR3044134}, we can approximate $\Sg$ by a family of smooth  surfaces $\Sg_j$ such that $N_j$ converges to $N$, $(\nu_h)_j$ converges to $\nu_h$ and $\{Z_j, S_j\}$ converge to $\{Z,S\}$, in any compact subset of $\Sg$, where the subscript $j$ denotes the vectors in the surface $\Sg_j$. Furthermore, we also have the convergence of the derivatives in the $Z_j$-direction of $Z_j$ and $(\nu_h)_j$ to the derivatives in the $Z$-direction of $Z$ and $\nu_h$ a on compact subsets of $\Sg$. In this way we can easily conclude the statement.
\end{proof}

\begin{remark} Proposition \ref{1var:weakstationary} and the fact that characteristic curves are horizontal straight lines imply that the mean curvature $H$ vanishes in any regular point of an area-stationary surface $\Sg$.
\end{remark}

The following Lemma will be needed in the next section.

\begin{lemma}[Divergence Theorem]
\label{lem:divergence}
Let $\Sg\subset\hh^1$ be an immersed oriented $C^1$ surface and $f,g \in C^0(\Sg\setminus\Sg_0)$ with compact support such that $f$ is differentiable in the horizontal direction and $g$ twice differentiable in the horizontal direction. Assume $Z$ and $\nu_h$ are also smooth in the horizontal direction. Then we have
\begin{equation}
\label{eq:div}
\int_\Sg \big\{Z(f)\,Z(g)+f\,Z(Z(g))+2\,\frac{\escpr{N,T}}{|N_h|}\,f\,Z(g)\big\}\,(|N_h|d\Sg)=0.
\end{equation}
\end{lemma}

\begin{proof} When $\Sg$ is $C^2$, it makes sense to compute from \cite[Lemma~3.1 and Lemma~3.3]{MR3044134}:
\begin{align*}
\divv_\Sg (\mnh f Z(g)Z)&= Z(\mnh) f Z(g)+\mnh Z(f)Z(g)+\mnh f Z(Z(g)) +2\nt\mnh^2 f Z(g)\\
&\qquad+2\nt^3fZ(g)+\ntnh Z(\nt)fZ(g)\\
&=\mnh Z(f)Z(g)+\mnh f Z(Z(g)) +2\nt f Z(g).
\end{align*}
The general case follows by approximation as in the proof of \cite[Proposition 6.3]{MR3044134}.
\end{proof}

\section{The second variation formula for area-stationary surfaces}\label{sec:2ndvariation}

In this Section, we shall prove the following second variation formula for area-stationary $C^1$ surface

\begin{theorem}
\label{thm:second}
Let $\Sg\subset\hh^1$ be a complete oriented area-stationary surface of class $C^1$, $U$ a horizontal $C^2$ vector field with compact support in $\hh^1$ and $\{\varphi_s\}_{s\in\rr}$ its associated one-parameter group of diffeomorphisms.
Then the second derivative of the sub-Riemannian area of the variation induced by $U$ is given by
\begin{equation}
\label{eq:2ndvar}
\frac{d^2}{ds^2}\Big|_{s=0}A(\varphi_s(\Sg))=\int_\Sg \big\{Z(f)^2-qf^2\big\}\, (|N_h|\,d\Sg),
\end{equation}
where $f=\escpr{U,\nu_h}$ and 
\begin{equation}
\label{eq:q}
q:=4\,\bigg(Z\bigg(\frac{\escpr{N,T}}{|N_h|}\bigg)+\frac{\escpr{N,T}^2}{|N_h|^2}\bigg).
\end{equation}
\end{theorem}

Following the classical terminology, we shall say that a complete oriented area-stationary surface $\Sg\subset\hh^1$ is \emph{stable} if
\[
\int_\Sg \big\{Z(f)^2-qf^2\big\}\, (|N_h|\,d\Sg)\ge 0
\]
holds for any continuous function $f$ on $\Sg$ with compact support such that $Z(f)$ exits and is continuous. The stability condition means that the second derivative of the sub-Riemannian perimeter is non-negative for the given variations.

For a $C^1$ surface $\Sg\subset\hh^1$ is not guaranteed that $q$ is well defined. We need to show first that the quantity $u=\escpr{N,T}/|N_h|$ is smooth along horizontal lines.% For smooth $\Sg$ we have that this quantity satisfies a non linear ordinary differential equation

\begin{lemma}
\label{lem:uab}
Given $a$, $b\in\rr$, the only solution of equation
\begin{equation}
\label{eq:codazzi}
u''+6\,u'u+4\,u^3=0,
\end{equation} 
about the origin with initial~conditions $u(0)=a$, $u'(0)=b$, is
\begin{equation}
\label{eq:codazzi-solutions}
u_{a,b}(s):=\frac{a+(2 a^2+b)\, s}{1+2as+(2 a^2+b)\, s^2}.
\end{equation}

If $u_{a,b}(s)$ is defined for every $s\in\rr$ then either $a^2+b>0$ or $u_{a,b}(s)\equiv 0$. Moreover
\begin{equation}
\label{eq:u'+u^2}
u_{a,b}'(s)+u_{a,b}(s)^2=\frac{a^2+b}{(1+2as+(2a^2+b)\,s^2)^2}.
\end{equation}
\end{lemma}

\begin{proof}
The first part of the Lemma is just a direct computation using the uniqueness of solutions for ordinary differential equations. For the second part, let us write $u_{a,b}(s)=p(s)/r(s)$, where $p(s)=a+(2a^2+b)\,s$, $r(s)=1+2as+(2a^2+b)\,s^2$. If $r(s)$ has no real zeroes, then the discriminant $-4a^2-4b$ of the polynomial $1+2as+(2a^2+b)\,s^2$ must be negative. If $r(s)$ has a zero $s_0$, then we must have $p(s_0)=0$ in order to have $u_{a,b}(s)$ well defined at $s_0$. Hence $u_{a,b}(s)$ can be expressed as the quotient of a constant over a degree one polynomial. From \eqref{eq:codazzi-solutions} we get $u_{a,b}=a\,(1+2as)^{-1}$, which has a pole unless $a=0$, i.e., $u_{a,b}(s)\equiv 0$. Equation \eqref{eq:u'+u^2} is just a direct computation.
\end{proof}

\begin{lemma}
\label{lem:continuity}
Consider $a,b,a_i,b_i\in\rr$, $i\in\nn$, so that $u_{a_i,b_i}$, $u_{a,b}$ are defined on the whole real line and the sequence $u_{a_i,b_i}$ converges pointwise to $u_{a,b}$. Then $a=\lim_{i\to\infty} a_i$, $b=\lim_{i\to\infty} b_i$. 
\end{lemma}

\begin{proof}
We first observe $\lim_{i\to\infty} a_i=\lim_{i\to\infty} u_{a_i,b_i}(0)=u_{a,b}(0)=a$. For the second limit, we remark that the sequence $b_i$ cannot have $\infty$ as an accumulation point since in this case $u_{a_i,b_i}(s)$ would converge to $1/s$ near $0$. Assume $c$ is a finite accumulation point of the sequence $b_i$, obtained as the limit, when $j\to\infty$, of the subsequence $b_{i_j}$. We first consider the case $a\neq 0$. Evaluating at $s=-1/(2a)$, we get
\[
-\frac{2ab}{2a^2+b}=u_{a,b}(-1/(2a))=\lim_{j\to\infty} u_{a_{i_j},b_{i_j}}(-1/(2a))
%=-\lim_{i\to\infty}\frac{2a\,(2a_{i_j}^2+b_{i_j}^2-2aa_{i_j})}{2a_{i_j}^2+b_{i_j}-4aa_{i_j}+4a^2}
=-\frac{2ac}{2a^2+c},
\]
what implies $b=c$. Hence $b=\lim_{i\to\infty} b_i$. In case $a=0$, 
\[
\frac{b}{1+b}=u_{0,b}(1)=\lim_{j\to\infty} u_{a_{i_j},b_{i_j}}(1)=\lim_{j\to\infty}\frac{a_{i_j}+(2a_{i_j}^2+b_{i_j})}{1+2a_{i_j}+(2a_{i_j}^2+b_{i_j})}=\frac{c}{1+c},
\]
what implies $b=c$. Hence $b=\lim_{i\to\infty} b_i$.
\end{proof}

\begin{lemma}
\label{lem:parameterization}
Let $\Sg\subset\hh^1$ be a complete oriented area-stationary surface of class $C^1$ without singular points, and let $\Ga:I\to\Sg$ be an integral curve of $S$. Define the parameterization $F(\eps,s)=\Ga(\eps)+s\,Z_{\Ga(\eps)}$ and let $\ga_\eps(s)=F(\eps,s)$. Then we have
\begin{enumerate}
\item $Z_{\Ga(\eps)}$ is smooth for almost every $\eps\in I$.
\item If $\eps$ is a regular point of $Z_{\Ga(\eps)}$, then $F$ is differentiable at $(\eps,s)$ for all $s\in\rr$. Moreover, the vector fields $\ptl F/\ptl s$ and $\ptl F/\ptl\eps$ are orthogonal with respect to the left-invariant Riemannian metric and
\begin{equation}
\label{eq:jacobian}
\Jac(F)(\eps,s)=|V_\eps(s)|=|N_h(\ga_\eps(s))|^{-1}|\escpr{V_\eps(s),T_{\ga_\eps(s)}}|, 
\end{equation}
where $V_\eps:=\ptl F/\ptl \eps$.
\item Along any geodesic $\ga_\eps(s)$, the function $u_\eps(s)=(\escpr{N,T}/|N_h|)(\ga_\eps(s))$ satisfies the ordinary differential equation \eqref{eq:codazzi}.
\item If $\eps$ is a regular point of $Z_{\Ga(\eps)}$, then $\escpr{V_\eps(s),T_{\ga_\eps(s)}}=a_\eps+b_\eps s+c_\eps s^2$, where
\begin{equation}
\label{eq:abceps}
a_\eps=-|N_h|(\Ga(\eps)),\quad b_\eps=-2\escpr{N,T}(\Ga(\eps)), \quad c_\eps=-|N_h|\,\bigg(Z\bigg(\frac{\escpr{N,T}}{|N_h|}\bigg)+2\,\frac{\escpr{N,T}^2}{|N_h|^2}\bigg)(\Ga(\eps)).
\end{equation}
\end{enumerate}
\end{lemma}

\begin{proof}
1. Fix some $\eps_0\in I$. For some $s_0>0$, let us take a seed curve $\Ga_1$ satisfying $\Ga_1(\eps_0)=\ga(s_0)$. A continuity argument implies the existence of a small open interval $J\subset I$ containing $\eps_0$, and functions $s:J\to\rr$ and $f:J\to\rr$ so that $f$ is monotone~increasing and
\[
\Ga(\eps)+s(\eps)\,Z_{\Ga(\eps)}=\Ga_1(f(\eps)).
\]
Hence $s(\eps)\,Z_{\Ga(\eps)}=\Ga_1(f(\eps))-\Ga(\eps)$ is smooth for almost every $\eps$ in a small interval near $0$. Dividing $s(\eps)\,Z_{\Ga(\eps)}$ by its Heisenberg modulus, which only involves the coordinates of $Z_{\Ga(\eps)}$ and the smooth components of the curve $\Ga(\eps)$, we conclude that $Z_{\Ga(\eps)}$ is smooth for almost every $\eps>0$ in $J$. We conclude by covering $I$ by a denumerable family of such intervals.

2. Let $\Ga(\eps)=(x(\eps),y(\eps),t(\eps))$, and $Z_{\Ga(\eps)}=(a(\eps),b(\eps),c(\eps))$, so that
\[
F(\eps,s)=(x(\eps)+sa(\eps),y(\eps)+sb(\eps),t(\eps)+s c(\eps)).
\]
Observe that $(c-ay+bx)(\eps)=0$ since $Z_{\Ga(\eps)}$ is horizontal. We always have\begin{align*}
\frac{\ptl F}{\ptl s}(\eps,s)&=a(\eps)\,X_{F(\eps,s)}+b(\eps)\,Y_{F(\eps,s)}.
\end{align*}
If $Z_{\Ga(\eps)}$ is smooth at some value $\eps$, we can compute
\begin{equation}
\label{eq:veps}
\begin{split}
V_\eps(s)=\frac{\ptl F}{\ptl\eps}(\eps,s)&=\big(x'(\eps)+sa'(\eps)\big)\,X_{F(\eps,s)}+\big(y'(\eps)+sb'(\eps)\big)\,Y_{F(\eps,s)}
\\
&+\big((t'-x'y+xy')(\eps)+2\,(ay'-bx')(\eps)\,s+(ab'-a'b)(\eps)\,s^2\big)\,T_{F(\eps,s)}.
\end{split}
\end{equation}
The scalar product of $\ptl F/\ptl s$ and $\ptl F/\ptl\eps$ can be easily seen to be equal to $(ax'+by')(\eps)+(aa'+bb')(\eps)\,s$, which is equal to $0$ since $\Ga'(\eps)=S_{\Ga(\eps)}$ is orthogonal to $Z_{\Ga(\eps)}$ and $|Z_{\Ga(\eps)}|=1$. From these properties equality $\Jac(F)(\eps,s)=|V_\eps(s)|$ in \eqref{eq:jacobian} follows.

Since $V_\eps(s)$ is orthogonal to $Z_{\ga_\eps(s)}$, there exists a function $f$ so that $V_\eps=fS$ and so $|V_\eps|=|f|$. On the other hand $\escpr{V_\eps,T}=-f\,|N_h|$, from where the last equation of \eqref{eq:jacobian} follows.

3. Assume first that $\eps$ is a regular point of $Z_{\Ga(\eps)}$. The vector $V_\eps(s)$ is proportional to $S_{F(\eps,s)}$ since it is orthogonal to $(\ptl F/\ptl s)(\eps,s)=Z_{F(\eps,s)}$. Hence we have
\begin{align*}
\frac{\escpr{N,T}}{|N_h|}(F(\eps,s))&=\frac{\escpr{V_\eps(s),J(Z_{F(\eps,s)})}}{\escpr{V_\eps(s),T_{\ga_\eps(s)}}}
\\
&=\frac{(ay'-bx')(\eps)+(ab'-a'b)(\eps)\,s}{(t'-x'y+xy')(\eps)+2\,(ay'-bx')(\eps)\,s+(ab'-a'b)(\eps)\,s^2},
\end{align*}
which is the function $u_{\bar{a},\bar{b}}(s)$, for
\begin{equation}
\label{eq:barabarb}
\begin{split}
\bar{a}&=\frac{(ay'-bx')(\eps)}{(t'-x'y+xy')(\eps)},
\\
2\bar{a}^2+\bar{b}&=\frac{(ab'-a'b)(\eps)}{(t'-x'y+xy')(\eps)},
\end{split}
\end{equation}
and hence it satisfies equation \eqref{eq:codazzi}.

For arbitrary $\eps$, consider a sequence of regular values $\eps_i$ of $Z_{\Ga(\eps)}$ with $\lim_{i\to\infty}\eps_i=\eps$.  For every $i$, there exist $a_i,b_i\in\rr$ so that $u_{a_i,b_i}(s)=(\escpr{N,T}/|N_h|)(\ga_{\eps_i}(s))$. Let $u(s)=(\escpr{N,T}/|N_h|)(\ga_\eps(s))$. Since $u_i\to u$ pointwise, we have $a_i=u_i(0)\to u(0)= a$. If the sequence $\{b_i\}_{i\in\nn}$ were unbounded, then $u_{a_i,b_i}$ should converge to $1/s$ near $0$, which is clearly unbounded and yields a contradiction. Hence $b_i$ is bounded, we can extract a convergent subsequence of $u_{a_i,b_i}$ to some $u_{a,b}$, and the function $u(s)=u_{a,b}(s)$ is smooth along the horizontal line $\ga(s)$ and satisfies the ordinary differential equation \eqref{eq:codazzi}.

4. Observe that $\escpr{V_\eps(s),T_{\ga_\eps(s)}}$ can be computed from \eqref{eq:veps} with
\[
a_\eps=(t'-x'y+xy')(\eps),\quad b_\eps=2\,(ay'-bx')(\eps),\quad c_\eps=(ab'-a'b)(\eps).
\]
We immediately obtain
\begin{align*}
a_\eps&=\escpr{S_{\Ga(\eps)},T_{\Ga(\eps)}}=-|N_h|(\Ga(\eps)),
\\
b_\eps&=2\,\escpr{S_{\Ga(\eps},J(Z_{\Ga(\eps)})}=-2\,\escpr{N,T}(\Ga(\eps)).
\end{align*}
For $c_\eps$, we have from \eqref{eq:barabarb}
\[
\bigg(Z\bigg(\frac{\escpr{N,T}}{|N_h|}\bigg)+2\,\frac{\escpr{N,T}^2}{|N_h|^2}\bigg)(\Ga(\eps))=\frac{(ab'-a'b)(\eps)}{-|N_h|(\Ga(\eps))}.
\]
\end{proof}

\begin{lemma}
\label{lem:p}
Let $\Sg\subset\hh^1$ be a complete oriented area-stationary surface of class $C^1$. Consider a characteristic geodesic $\ga:\rr\to\Sg$. Let $p(s)=a+bs+cs^2$, where
\[
a=-|N_h|(\ga(0)),\quad b=-2\escpr{N,T}(\ga(0)), \quad c=-|N_h|\,\bigg(Z\bigg(\frac{\escpr{N,T}}{|N_h|}\bigg)+2\,\frac{\escpr{N,T}^2}{|N_h|^2}\bigg)(\ga(0)).
\]
If $(q\circ\ga)(s)$ is not identically zero, then $v(s)=|p(s)|^{1/2}$ is positive everywhere and satisfies
\begin{equation}
\label{eq:v}
((v(s)^{-1})')^2-\frac{1}{2}\,(v(s)^{-2})''-\frac{\escpr{N,T}}{|N_h|}(\ga(s))(v(s)^{-2})'=\frac{q(\ga(s))}{4v(s)^2}.
\end{equation}
\end{lemma}

\begin{proof}
Since $p(s)$ is a degree two polynomial, it has a zero if and only if $b^2-4ac<0$. A simple computation shows
\begin{equation}
\label{eq:b^2-4ac}
b^2-4ac=-4\,|N_h|\,\bigg(Z\bigg(\frac{\escpr{N,T}}{|N_h|}\bigg)+\frac{\escpr{N,T}^2}{|N_h|^2}\bigg)(\ga(0)),
\end{equation}
which is negative by Lemma~\ref{lem:uab} and \eqref{eq:u'+u^2}.

To check equation \eqref{eq:v}, we compute $(\escpr{N,T}/|N_h|)(\ga(s))$ in terms of $a$, $b$, $c$, to obtain
\[
\frac{\escpr{N,T}}{|N_h|}(\ga(s))=\frac{(b/2)+cs}{a+bs+cs^2}.
\]
A straightforward computation of the left side of \eqref{eq:v}, together with \eqref{eq:b^2-4ac} and \eqref{eq:u'+u^2}, yields
\[
((v(s)^{-1})')^2-\frac{1}{2}\,(v(s)^{-2})''-\frac{\escpr{N,T}}{|N_h|}(\ga(s))(v(s)^{-2})'=\frac{b^2-4ac}{4\,(a+bs+cs^2)^3}=\frac{q(\ga(s))}{4\,v(s)^2},
\]
what implies \eqref{eq:v}.
\end{proof}

\begin{lemma}
\label{lem:codazzi}
Let $\Sg\subset\hh^1$ be a complete oriented area-stationary surface of class $C^1$, and $\ga:\rr\to\Sg$ be a characteristic geodesic parameterized by arc-length. %Let $u(s)=(\escpr{N,T}/|N_h|)(\ga(s))$. Then
\begin{enumerate}
%\item There exists $a$, $b\in\rr$ such that $u(s)=u_{a,b}(s)$, where the function $u_{a,b}(s)$ satisfies equation \eqref{eq:codazzi-solutions}. In particular, $u$ satisfies the ordinary differential equation \eqref{eq:codazzi}.
\item Either $q(\ga(s))> 0$ for all $s\in\rr$ or $q(\ga(s))\equiv 0$. If $q(\ga(s_0))=0$, then $T_{\ga(s_0)}$ is tangent to $\Sg$ at $\ga(s_0)$.
\item If $q\equiv 0$ on $\Sg$ then $\Sg$ is a vertical plane.
%\item $u'(s)+u(s)^2\ge 0$ for all $s\in\rr$, and $u'(s)+u(s)^2=0$ if and only if ...
\end{enumerate}
\end{lemma}

\begin{proof}
To prove 1, observe that $q(\ga(s))=4\,(u'(s)+u(s)^2)$. Equation \eqref{eq:u'+u^2} implies
\begin{equation*}
u'(s)+u(s)^2=\frac{a^2+b}{(1+2as+(2a^2+b)\,s^2)^2},
\end{equation*}
which is either positive or identically zero by Lemma~\ref{lem:uab}. If $q(\ga(s_0))=0$ then $\escpr{N,T}(\ga(s_0))=0$ and so $T_{\ga(s_0)}$ is tangent to $\Sg$.

To prove 2, represent $\Sg$ locally as the intrinsic graph of a function $u$ of class $C^1$ over an open domain of the plane $y=0$. Since $T$ is tangent to $\Sg$ we have $u_t\equiv 0$. As $\Sg$ is foliated by characteristic geodesics, $u_x$ is constant along the characteristic curves $t'=2ux'$. Together with condition $u_t\equiv 0$, this implies that $u_x$ is constant and so $u(x,t)=ax$, for some real constant $a$. This implies that $\Sg$ is locally the image of $(x,t)\mapsto (x,ax,t-ax^2)$, which is part of the  vertical plane $y=ax$.
\end{proof}

\begin{remark}
\label{rem:codazzi}
Given the $t$-graph of a $C^1$ function $u:\Om\subset\rr^2\to\rr$, Cheng, Hwang and Yang proved in \cite[Thm.~A]{MR2983199} that the function $D=((u_x-y)^2+(u_y+x)^2)^{1/2}$ satisfies the diferential equation
\begin{equation}
\label{eq:D}
DD''=2(D'-1)(D'-2),
\end{equation}
where the prime symbol $'$ denotes the derivative along the projection to $t=0$ of the characteristic vector field in the surface. Since the Riemannian unit normal of the $t$-graph of $u$ is given by
\[
N=\frac{(u_x-y)\,X+(u_y+x)\,Y-T}{(1+(u_x-y)^2+(u_y+x)^2)^{1/2}},
\]
one can immediately check that $D=|N_h|/\escpr{N,T}=u^{-1}$ (in a $t$-graph, the Reeb vector field $T$ is never tangent to the surface). A simple computation then shows that \eqref{eq:codazzi} implies \eqref{eq:D}. Equation~\ref{eq:D} played a key role in \cite{MR2983199} to obtain strong structure results of the singular set of an area-stationary $C^1$ surface and, in general, for weak solutions of the mean curvature equation.
\end{remark}

\begin{proof}[Proof of Theorem~\ref{thm:second}] For every $p\in\Sg$,
we consider extensions $E_1(s), E_2(s)$ of $Z_p$, $S_p$ along the curve $s\mapsto \varphi_s(p)$ so that $[E_i,U] = 0$, i.e., the vector fields $E_i$ are invariant under the flow generated by $U$. They are smooth with respect to $U$.

Denoting by  $V=\escpr{E_1,T}\,E_2-\escpr{E_2,T}\,E_1$ the horizontal Jacobian, see \cite[Lemma~3.2]{MR3044134}, we get from \cite[Lemma~3.4]{MR3044134} that
\[
A(\varphi_s(\Sg))=\int_\Sg |V(s)| \, d\Sg
\]
and
\[
\frac{d^2}{ds^2}\Big|_{s=0}A(\varphi_s(\Sg))=\int_\Sg U(U(|V|))\, d\Sg,
\]
as only $C^1$ regularity of the surface is required.
Since
\begin{equation*}
\begin{split}
U(U(|V|))=&-\frac{1}{|V|^3}\escpr{\n_{U}V,V}^2 +\frac{1}{|V|}\left(\escpr{\n_{U}\n_{U}V,V}+|\n_{U}V|^2\right)\\
&=\frac{1}{\mnh}\left(  \mnh \escpr{\n_{U}\n_{U}V,Z}+ \escpr{\n_U V,\nu_h}^2 +\escpr{\nabla_UV,T}^2\right),
\end{split}
\end{equation*}
at $s=0$, we only need to compute the expressions $\escpr{\n_{U}\n_{U}V,Z}$, $\escpr{\n_U V,\nu_h}^2$, and $\escpr{\nabla_UV,T}^2$ along the curve $s\mapsto\varphi_s(p)$ and evaluate them at $s=0$, corresponding to the point $p$.
We have 
\begin{equation*}
\nabla_UV=2\,\escpr{J(U),Z}\escpr{N,T}\nuh-2\,\escpr{J(U),S}Z+\mnh\nabla_ZU,
\end{equation*}
at $p$, what implies $\escpr{\nabla_UV,T}=0$, and
\begin{equation*}
\escpr{\n_U V,\nu_h}=2\,\escpr{U,\nuh}\escpr{N,T}+\mnh\,Z(\escpr{U,\nuh}).
\end{equation*}
On the other hand 
\begin{align*}
\escpr{\nabla_U\nabla_U V,Z}&=2\,\escpr{\n_UE_1,T}\escpr{\n_U E_2,Z}-2\,\escpr{\n_U E_2,T}\escpr{\n_U E_1,Z}
\\
&\qquad +\mnh\,\escpr{\n_U\n_U E_1,Z}-\escpr{\n_U\n_U E_2,T},
\end{align*}
at $p$, since 
\begin{align*}
\escpr{\n_U E_i,T}&=2\,\escpr{J(U),E_i},\quad i=1,2,
\\
\escpr{\n_U E_i,Z}&=\escpr{\n_{E_i}U,Z},\quad i=1,2,
\\
\escpr{\n_U\n_U E_i,T}&=2\,\escpr{J(\n_UU),E_i}+2\,\escpr{J(U),\n_{E_i}U},\quad i=1,2,
\\
\escpr{\n_U\n_U E_i,Z}&=\escpr{\nabla_{E_i}(\nabla_UU),Z}, \quad i=1,2.
\end{align*}
The above formulas follow from $\nabla_UE_i=\nabla_{E_i}U+2\,\escpr{J(U),E_i}\,T$ since $[E_i,U]=0$ and from the facts that $U$ is horizontal and the curvature tensor associated to the connection $\nabla$ vanishes. Using these formulas we get
\begin{align*}
\escpr{\nabla_U\nabla_U V,Z}&=4\,\escpr{J(U),Z}\escpr{\nabla_SU,Z}-4\,\escpr{J(U),S}\escpr{\nabla_ZU,Z}
\\
&\qquad +\mnh\escpr{\nabla_Z(\nabla_UU),Z}-2\,\escpr{J(\n_UU),S}-2\,\escpr{J(U),\n_SU}.
\end{align*}

Observe that
\[
\int_\Sg \big\{-2\,\escpr{J(\n_UU),S}+\mnh\escpr{\nabla_Z(\nabla_UU),Z}\big\}\,d\Sg=0
\]
by the first variation formula \eqref{eq:1st-0} applied to the horizontal $C^1$ vector field $\n_UU$. Hence we only need to consider the integral over $\Sg$ of the remaining summands
\begin{multline}
\label{eq:nuuv}
\frac{1}{\mnh}\big(4\,\escpr{U,\nuh}^2\escpr{N,T}^2+4\,\mnh\escpr{N,T}\escpr{U,\nuh}\,Z(\escpr{U,\nuh})+\mnh^2\,Z(\escpr{U,\nuh})^2\big)
\\
+4\,\escpr{J(U),Z}\escpr{\nabla_SU,Z}-4\,\escpr{J(U),S}\escpr{\nabla_ZU,Z}-2\,\escpr{J(U),\n_SU}.
\end{multline}

Letting $f=\escpr{U,\nuh}$, the second term can be written as
\[
2\mnh\frac{\escpr{N,T}}{\mnh}\,Z(f^2)=2\mnh\,\bigg(Z\bigg(\ntnh f^2\bigg)-Z\bigg(\ntnh\bigg)\,f^2\bigg).
\]
Integration by parts ($\int_\Sg\big(Z(h)+2\tfrac{\escpr{N,T}}{\mnh}h\big)\,\mnh d\Sg=0$) implies that the integral of the first two terms in \eqref{eq:nuuv} is equal to
\begin{equation}
\label{eq:1}
\int_\Sg \bigg(4\frac{\escpr{N,T}^2}{\mnh^2}f^2+2\ntnh Z(f^2)\bigg)\,\mnh d\Sg=-2\int_\Sg Z\bigg(\ntnh\bigg) f^2\,(\mnh d\Sg).
\end{equation}

The integral in the third summand in \eqref{eq:nuuv} is equal to
\begin{equation}
\label{eq:2}
\int_\Sg Z(f)^2\,(\mnh d\Sg).
\end{equation}

The last three terms can be treated using Lemma~\ref{lem:approx}. We approximate $\Sg$ by a family of $C^\infty$ surfaces $\Sg_i$ uniformly on the support of $U$. Let $H_i$ be the mean curvature of $\Sg_i$. From \cite[Lemma~6.2]{MR3044134}, $H_i\to 0$ uniformly and $Z_i(\escpr{N_i,T}/|(N_i)_h|)$ uniformly converges to $Z(\escpr{N,T}/|N_h|$. Then the convergence of the integrals and Lemma~\ref{lem:approx} imply
\begin{multline}
\label{eq:3}
\int_\Sg \big\{4\,\escpr{J(U),Z}\escpr{\nabla_SU,Z}-4\,\escpr{J(U),S}\escpr{\nabla_ZU,Z}-2\,\escpr{J(U),\n_SU}\big\}\,\mnh d\Sg
\\
=-\int_\Sg\bigg(2Z\bigg(\frac{\escpr{N,T}}{\mnh}\bigg)+4\frac{\escpr{N,T}^2}{\mnh^2}\bigg)\,f^2\,\mnh d\Sg.
\end{multline}

Using \eqref{eq:1}, \eqref{eq:2} and \eqref{eq:3} to evaluate \eqref{eq:nuuv} we obtain the second variation formula \eqref{eq:2ndvar}.
\end{proof}

\begin{lemma}
\label{lem:approx}
Let $\Sg\subset\hh^1$ be an oriented surface of class $C^2$ with mean curvature $H$. Let $U$ be a $C^1$ vector field in $\hh^1$ with compact support so that $\text{supp}(U)\cap\Sg_0=\emptyset$, and let $f=\escpr{U,\nuh}$, $g=\escpr{U,Z}$. Then
\begin{multline}
\label{eq:approx}
\int_\Sg \big\{4\,\escpr{J(U),Z}\escpr{\nabla_SU,Z}-4\,\escpr{J(U),S}\escpr{\nabla_ZU,Z}-2\,\escpr{J(U),\n_SU}\big\}\,\mnh d\Sg
\\
=-\int_\Sg\bigg(2Z\bigg(\frac{\escpr{N,T}}{\mnh}\bigg)+4\frac{\escpr{N,T}^2}{\mnh^2}\bigg)\,f^2\,\mnh d\Sg+2\int_\Sg \ntnh Hfg\,(\mnh d\Sg).
\end{multline}
\end{lemma}

\begin{proof}
The integrand in the left integral of \eqref{eq:approx} is given by
%\begin{equation*}
%4f\big(f\escpr{\nabla_S\nuh,Z}+S(g)\big)+4\escpr{N,T}g\escpr{\nabla_ZU,Z}-2\big(fS(g)+f^2\escpr{\n_S\nuh,Z}-gS(f)-g^2\escpr{\n_SZ,\nuh}\big),
%\end{equation*}
%that it is equal to
\[
2f^2\escpr{\n_S\nuh,Z}+2 S(fg)+4\escpr{N,T}H(fg)+4\escpr{N,T}gZ(g) -2g^2\escpr{\n_S\nuh,Z},
\]
where $H$ is the mean curvature of the surface. Taking into account that
\[
\escpr{\n_S\nuh,Z}=-\mnh Z\bigg(\frac{\escpr{N,T}}{\mnh}\bigg)-2\mnh\frac{\escpr{N,T}^2}{\mnh^2},
\]
see \cite[Lemma~3.1 (v)]{MR3044134} or \cite[Lemma~3.11 (vii)-(viii)]{Gaphd}, we immediately see that the integrals of the last two summands equal
\[
\int_\Sg \bigg(2\ntnh Z(g^2)+2g^2 Z\bigg(\ntnh\bigg) +4\frac{\escpr{N,T}^2}{\mnh^2}g^2 \bigg) (\mnh d\Sg)=0.
\]
To treat the second and third summand we take into account formula $\int_\Sg (S(h)+h\escpr{N,T}H)\,d\Sg=0$, see \cite[Lemma~3.3]{MR3044134}, so that
\[
\int_\Sg \big(2S(fg)+4\escpr{N,T} H (fg)\big)\,d\Sg=2\int_\Sg \ntnh Hfg\,(\mnh d\Sg).
\]
This implies \eqref{eq:approx}.
\end{proof}

Now we introduce the stability operator $\mathcal{Q}(f,f)$.

\begin{theorem}
\label{thm:stabilityoperator}
Let $\Sg\subset\hh^1$ be a complete oriented area-stationary surface of class $C^1$. If $\Sg$ is stable, then the operator 
\begin{equation}
\label{eq:stabilityoperator}
\mathcal{Q}(f,f):=\int_\Sg \big\{Z(f)^2-qf^2\big\}\, (|N_h|\,d\Sg),
\end{equation}
is non-negative for all functions $f\in C^0(\Sg\setminus\Sg_0)$  with compact support of class $C^1$ in the horizontal direction in $\Sg$,
where $q$ is the function defined in \eqref{eq:q}.
\end{theorem}

\begin{proof}
We express $\nu_h$ as $\nu_h=gX+hY$, with $g,h$ continuous functions on $\Sg$, differentiable in the $Z$-direction and satisfying $g^2+h^2=1$. 
Let $\tilde{f},\tilde{g},\tilde{h}$ continuous extensions of $f,g,h$ in a neighborhood of $\Sg$ and  let $\tilde{f}_\eps,\tilde{g}_\eps,\tilde{h}_\eps$ the standard mollifiers of  $\tilde{f},\tilde{g},\tilde{h}$. 

Since the vector fields $U_\eps=f_\eps(g_\eps X+h_\eps Y)\in C^\infty_0(\Sg\setminus\Sg_0)$ are horizontal, the non-negativity of the second variation induced by $U_\eps$ implies 
\[
\begin{split}
0\leq & \lim\limits_{\eps\rightarrow 0}\int_\Sg \big\{Z(f_\eps(g_\eps g+h_\eps h))^2-q(f_\eps(g_\eps g+h_\eps h))^2\big\}\, (|N_h|\,d\Sg)\\
&= \int_\Sg \big\{Z(f ( g^2+h^2 ))^2-q(f(g^2+h^2))^2\big\}\, (|N_h|\,d\Sg)= \mathcal{Q}(f,f),
\end{split}
\]
that allows us to conclude $\mathcal{Q}(f,f)\geq 0$.
\end{proof}

\section{Stable $C^1$ surfaces without singular points}
\label{sec:second}

In this Section we shall prove the following result

\begin{theorem}
\label{thm:main}
Let $\Sg\subset\hh^1$ be a complete oriented stable surface of class $C^1$ without singular points. Then $\Sg$ is a vertical plane.
\end{theorem}

\begin{proof}
%{\color{red} To be completed}
We shall follow the proof of \cite[Thm.~4.7]{MR2609016}%
%({\color{red} Cite the different variations. It seems that the first time this argument appeared was in a proof by Garofalo et al. as cited in \cite{MR2604618}}. Here or better in the introduction)
. Inserting a test function of the form $uv^{-1}$ in the index form \eqref{eq:2ndvar} and using formula \eqref{eq:div} with $f=u^2$, g$=v^{-2}$ we get
\begin{equation}
\label{eq:indexform1}
\begin{split}
\mathcal{Q}(uv^{-1}&,uv^{-1})=\int_\Sg v^{-2}Z(u^2)\,(|N_h|\,d\Sg)
\\
&+\int_\Sg u^2\,\bigg\{Z(v^{-1})^2-\frac{1}{2}\,Z(Z(v^{-2}))-\frac{\escpr{N,T}}{|N_h|}\,Z(v^{-2})\bigg\}\,(|N_h|\,d\Sg)
\\
&-\int_{\Sg} q\,(uv^{-1})^2 (|N_h|\,d\Sg).
\end{split}
\end{equation}

Let $\Ga:I\to\Sg$ be a $C^1$ seed curve (integral curve of $S$) defined on an open interval $I\subset\rr$. Consider the parameterization $F:I\times\rr\to\Sg$ given by $F(\eps,s):=\Ga(\eps)+s\,Z_{\Ga(\eps)}$ introduced in Lemma~\ref{lem:parameterization}, and let $\ga_\eps(s)=F(\eps,s)$. 

For every $\eps$ we consider the polynomial $p_\eps(s)=p(\eps,s)=a_\eps+b_\eps s+c_\eps s^2$ defined in Lemma~\ref{lem:p}, with coefficients
\[
a_\eps=-|N_h|(\ga_\eps(0)),\quad b_\eps=-2\escpr{N,T}(\ga_\eps(0)), \quad c_\eps=-|N_h|\,\bigg(Z\bigg(\frac{\escpr{N,T}}{|N_h|}\bigg)+2\,\frac{\escpr{N,T}^2}{|N_h|^2}\bigg)(\ga_\eps(0)).
\]
The function $p(\eps,s)$ is continuous since the coefficients $a_\eps$, $b_\eps$. $c_\eps$ depend continuously on $\eps$. The first two ones because of the continuity of the unit normal vector $N$. The latter one because of Lemma~\ref{lem:parameterization}.

Henceforth we assume $q\circ F>0$ on $I\times\rr$. Lemma~\ref{lem:p} then implies $p_\eps<0$ for all $\eps$. We define the function $v_\eps(s)=v(\eps,s)=|p_\eps(s)|^{1/2}$. Lemma~\ref{lem:p} implies that $v_\eps$ satisfies equation \eqref{eq:v} on the geodesic $\ga=\ga_\eps$. For this particular function $v$, \eqref{eq:indexform1} and \eqref{eq:v} imply
\begin{equation}
\label{eq:indexform2}
\mathcal{Q}(uv^{-1},uv^{-1})=\int_\Sg v^{-2}\,\big\{Z(u^2)-\frac{3}{4}\,q\,u^2\big\}\,(|N_h|\,d\Sg).
\end{equation}

For almost every $\eps$, the Jacobian of $F$ equals $|V_\eps(s)|=|N_h(\ga_\eps(s))|^{-1}|\escpr{V_\eps(s),T_{\ga_\eps(s)}}|$ by Lemma~\ref{lem:parameterization}. Hence the sub-Riemannian area element on $\Sg$ takes the form
\begin{equation}
\label{eq:area-element}
|N_h|\,d\Sg =|\escpr{V_\eps(s),T_{\ga_\eps(s)}}|\,d\eps\,ds.
\end{equation}
From Lemma~\ref{lem:parameterization}(4), $|\escpr{V_\eps(s),T_{\ga_\eps(s)}}|=|v|$ for almost all $\eps$. Hence \eqref{eq:indexform2} can be written as
\begin{equation}
\label{eq:indexform3}
\mathcal{Q}(uv^{-1},uv^{-1})=\int_{I\times\rr}\bigg(\frac{\ptl u}{\ptl s}\bigg)^2d\eps\,ds-\frac{3}{4}\int_{I\times\rr} q\,u^2\,d\eps\,ds.
\end{equation}
Once we have obtained this second variation formula, we just make a choice of test function as in the proof of Theorem~4.7 in \cite[p,~591]{MR2609016}: take a non-negative $C^\infty$ function $\phi:I\to\rr$ with $\phi(0)>0$ and compact support contained inside a bounded interval $I'\subseteq I$.  Let $\ell:=\text{length}(I')$, and $M$ a positive constant so that $|\phi'(\eps)|\leq M$ for all $\eps\in I$. For any $k\in\mathbb{N}$ we define the function
\[
u_{k}(\eps,s):=\phi(\eps)\,\phi(s/k).
\]
so that $u_{k}\in C_{0}(I'\times kI')$, and $u_{k}$ is $C^\infty$ with respect to $s$.  By Fubini's Theorem
\[
\int_{I\times\rr}\left(\frac{\ptl u_{k}}{\ptl s}\right)^2d\eps\,ds=
\frac{1}{k^2}\left(\int_{I'}\phi(\eps)^2\,d\eps\right)
\,\left(\int_{kI'}\phi'(s/k)^2\,ds\right)\leq\frac{\ell M^2}{k}
\int_{I'}\phi(\eps)^2\,d\eps,
\]
which goes to $0$ when $k\to\infty$.  Note also that $\{u_{k}\}_{k\in\mathbb{N}}$ pointwise converges when $k\to\infty$ to $u(\eps,s)=\phi(0)\,\phi(\eps)$.  Since $q\ge 0$ by Lemma~\ref{lem:codazzi}(1), we can apply Fatou's lemma to obtain
\[
\liminf_{k\to\infty}\int_{I\times\rr}q\,u^2_{k}\,d\eps\,ds
\geq\int_{I\times\rr}q\,u^2\,d\eps\,ds.
\]
We conclude from \eqref{eq:indexform3} that
\begin{equation*}
\limsup_{k\to\infty}\,\mathcal{Q}(u_{k}v^{-1},u_{k}v^{-1})=
-\frac{3}{4}\,\liminf_{k\to\infty}\int_{I\times\rr}q\,u^2_{k}\,
d\eps\,ds\leq-\frac{3}{4}\int_{I\times\rr}q\,
u^2\,d\eps\,ds,
\end{equation*}
which is strictly negative by our assumption $q\circ F>0$ on $I\times\rr$. This way we obtain instability, a contradicting arising from the assumption $q\circ F>0$. We conclude $q\equiv 0$ on $\Sg$. By Lemma~\ref{lem:codazzi}(2), $\Sg$ is a vertical plane.
\end{proof}

A trivial consequence of Theorem~\ref{thm:main} is the following

\begin{corollary}[Bernstein Theorem for intrinsic graphs of $C^1$ functions]
\label{cor:bernstein}
Let $\Sg\subset\hh^1$ be a complete locally area-minimizing intrinsic graph of a $C^1$ function. Then $\Sg$ is a vertical plane.
\end{corollary}

\begin{remark}
There are many examples of $t$-graphs of class $C^1$ with singular points which are area-minimizing in the Heisenberg group $\hh^1$ \cite{MR2448649}.
\end{remark}

\begin{remark}
Monti, Serra-Casanno and Vittone \cite{MR2455341} gave an example of a complete area-minimi\-zing $C^1$ surface, which is an intrinsic graph except over a vertical plane  of a function $C^\infty$ out of a given line, and it is not a vertical plane. Such a surface is one of the examples of area-minimizing $t$-graphs given in \cite{MR2448649}: $u_\la(x,y)=xy-\la^{-1}x|x|$, $\la>0$. It can be easily proved that such surfaces are the intrinsic graphs of the functions $v_\la(y,t)=\text{sgn}(t)\,\sqrt{\la|t|}$ of class $C^{0,1/2}$ over the vertical plane $x=0$. The graph of $v_\la$ is not $\hh$-regular since its horizontal unit normal is not continuous at $t=0$.
\end{remark}

\bibliography{stablec1}

\end{document}